\documentclass[12pt,a4paper]{article}
\usepackage{amsmath,amsthm,amsfonts,latexsym,amscd,amssymb, makeidx,graphics}
\usepackage[dvips]{color}

\makeindex

\newtheorem{theorem}{Theorem}[section]
\newtheorem{lemma}[theorem]{Lemma}
\newtheorem{corollary}[theorem]{Corollary}
\newtheorem{proposition}[theorem]{Proposition}

\theoremstyle{definition}

\newtheorem{example}[theorem]{Example}

\theoremstyle{remark}
\newtheorem{remark}[theorem]{Remark}

\newcommand{\N}{\mathbb{N}}

\newcommand{\Z}{\mathbb{Z}}

\newcommand{\n}{\noindent}
\newcommand{\g}{{\mathcal G}_1}

\def\addsymbol #1: #2#3{$#1$ \> \parbox{5in}{#2 \dotfill \pageref{#3}}\\}

\begin{document}

\begin{center}
\Large{On The Isomorphism Classes Of Transversals III} 
\end{center}

\begin{center}
 Vivek Kumar Jain \\
 Central University of Bihar, Patna, India\\
 Email: jaijinenedra@gmail.com
\end{center}
\large{\textbf{Abstract:}}
Let $G$ be a finite group and $H$ a subgroup of $G$. Each left transversal (with identity) of $H$ in $G$ has a left loop (left quasigroup with identity) structure induced by the binary operation of $G$.
We say two left transversals are isomorphic if they are isomorphic with respect to the induced left loop structures. In this paper, we develop a method to calculate the number of isomorphism classes of transversals of $H$ in $G$. Also with the help of this we calculate the number of non-isomorphic left loops of a given order.

\vspace{1cm}

\smallskip


\section{Introduction}
Let $H$ be a subgroup of a finite group $G$. A left transversal of $H$ in $G$ is a set of coset representative from each left coset of $H$ in $G$ with identity element (of $G$) from coset $H$. We denote the set of all left transversals of $H$ in $G$ by ${\cal T}(G,H)$. Let $S \in  {\cal T}(G,H)$. Then we can define a binary operation $\circ $ on $S$ as: for $x,y \in S$, $\{ x \circ y \}:= S \cap xyH$.

The pair $(S, \circ )$ is a groupoid with equations of the form $a \circ X=b$, $X$ is unknown and   $a, b \in S$ are solvable and $(S,\circ)$ has two sided identity. Such algebraic system $(S, \circ )$ is called a left loop or a left quasigroup with identity  and $\circ$ is called the induced binary operation. Same construction can be done for a right transversal with identity. A right transversal with respect to induced operation will form a right loop. A bijective map between two elements of ${\cal T}(G,H)$ is called isomorphism if it preserves induced binary operations. Thus ``being isomorphic$"$ is an equivalence relation on ${\cal T}(G,H)$
and the equivalence classes of ${\cal T}(G,H)$ under this relation are called \textit{isomorphism classes of left transversals}. Similarly, we can define the \textit{isomorphism classes of right transversals}. Lemma \ref{ict} shows that the number of  isomorphism classes of left transversals is equal to the number of isomorphism classes of right transversals. So we call it the  number of isomorphism classes of transversals of $H$ in $G$ and  denote it by \textbf{ict$(G,H)$}.
There are natural problems regarding this number, for example (i) On what properties of the pair $(G,H)$ does ict$(G,H)$ depends upon?
(ii) What types of information about group $G$ and subgroup $H$ can be deduced from this number?
(iii) What are the natural numbers which appear as ict$(G,H)$?   
(iv) How to calculate ict$(G,H)$ for a given pair $(G,H)$?

We do not know exact answer of (i). Following facts give partial answers of (ii) and (iii).

\n \textbf{Fact 1:} {\cite{pss}} ict$(G,H)=1$ if and only if $ H \trianglelefteq  G$.

\n \textbf{Fact 2:} {\cite{viv, vip}} ict$(G,H) \neq 2,4$.

\n \textbf{Fact 3:}\cite{ict2} ict$(G,H)=3$ if and only if $H  \not \trianglelefteq G$ and $[G:H]=3$.

\n \textbf{Fact 4:} \cite[Theorem 3.7, p. 2693]{rps} ict$(\text{Sym}(n), \text{Sym}(n-1))$ is the number of non-isomorphic right loops of order $n$, where $\text{Sym}(n)$ denotes the symmetric group on $n$ symbols.

These four facts shows the importance of the number $\text{ict}(G,H)$. But it is hard to determine ict$(G,H)$ for a given pair $(G,H)$.
In Section 2, we develop a method to calculate the number ict$(G,H)$ for the pair $(G,H)$ under some conditions. In Section 3, 4 and 5, we apply the method to calculate ict$(G,H)$ for $(G,H)$ equal to $(\text{Sym}(n), \text{Sym}(n-1))$, $(\text{Alt}(n), \text{Alt}(n-1))$ and $(D_n, \langle b \rangle)$ where $ \text{Alt}(n)$ and $D_n$ respectively denote the Alternating group on $n$ symbols and Dihedral group with $2n$ elements, and $\langle b \rangle $ denotes a non-normal subgroup of $D_n$ of order $2$ generated by $b \in D_n$. We will also calculate the isomorphism classes of transversals of a subgroup of order $p$ (prime) in a non-abelian group of order $pq$ ($q$ prime, $q > p$).

\section{Basic Ideas}

Let $H$ be a subgroup of index $n $ of a finite group $G$. The basic theory of right transversals is developed in \cite{rltr, rps}. A parallel theory can be developed for left transversals. Due to following lemma, we can restrict ourself to left transversals for determining isomorphism classes of transversals.
\begin{lemma}\label{ict}
The number of isomorphism classes of left transversals is  equal to the number of isomorphism classes of right transversals.
\end{lemma}
\begin{proof}
Let $S=\{x_1,x_2, \ldots , x_n \}$ and $T$ be two right transversals of $H$ in $G$ which are isomorphic via map $\sigma$. That is  $T=\sigma (S)$. Then $L_1=S^{-1}=\{x_1^{-1}, \ldots , x_n^{-1} \} $ and $L_2=\sigma(S)^{-1}$ are left transverals of $H$ in $G$. 
Define a map $\phi:L_1 \rightarrow L_2$ as $\phi (x_i^{-1})=\sigma (x_i)^{-1}$. 
It is easy to check that $\phi$ is a left loop isomorphism.
\end{proof}

Consider the map $\chi$ of permutation representation of $G$ on $G/^lH$, the set of all left cosets of $H$ in $G$, that is $\chi : G \rightarrow \text{Sym}(G/^l H)$ defined as $\chi (g) (xH)=gxH$. Identify $\text{Sym}(G/^l H)$ with $\text{Sym}(n)$ by putting unique number from the set $\{1,2, \ldots , n\} $ for each cosets and $1$ for the coset $H$. Note that $\chi$ is a group homomorphism with kernel equal to $Core_GH$, the core of $H$ in $G$ and it maps each left transversal of $H$ in $G$ isomorphically onto the left transversal of $\chi(H)$ in $\chi(G)$. Also $\chi$ is a surjective map from $ {\mathcal T}(G,H)$ to $ {\mathcal T}(\chi (G),\chi (H))$. Thus $\text{ict}(G,H)=\text{ict}(\chi(G),\chi(H))$. For calculation of $\text{ict}(G,H)$, there is no harm if we identify pair $(G,H)$ with the pair $(\chi(G),\chi(H))$ inside $\text{Sym}(n)$. Unless otherwise mentioned, \textit{by the pair $(G,H)$} we mean \textit{the pair $(\chi(G),\chi(H))$}. Now, $G$ is a transitive subgroup of $\text{Sym}(n)$ and $H$ is the set of stabilizer of symbol $1$. We denote by $\text{Sym}(n-1)$ the subgroup of $\text{Sym}(n)$ containing those permutations which fix $1$.
 
\begin{lemma} \label{1}
Let $T, L \in {\cal T}(G,H)$. Then a map $\sigma:T \rightarrow L$ is an isomorphism if and only if there exists $\alpha \in \text{Sym}(n-1)$ such that $\alpha T \alpha^{-1}=L$. Moreover, $\sigma = i_\alpha |_T$, the restriction of inner automorphism determined by $\alpha$ on $T$.
\end{lemma}

\begin{proof}
Suppose that $\sigma$ is an isomorphism from $T$ to $L$
and $G/^lH$ denotes the set of left cosets of $H$ in $G$.

Take $T=\{ x_1,x_2, \ldots , x_n \}$ and $L=\{ y_1, y_2,$ $ \ldots , $ $ y_n \}$ such that the suffixes of $x$ and $y$ denote the assigned number for their coset. Define $\alpha$, an element of $\text{Sym}(n)$ as $\alpha(i) =j$ if $\sigma(x_i)H=y_jH=x_jH$. Since $\sigma$ is a left quasigroup homomorphism,  $\sigma (x_1)=y_1$, that is $\alpha(1)=1$. 

We claim that for $x_k \in T$, $\alpha x_k \alpha ^{-1}= \sigma (x_k)$. 
Let $i,j \in \{1,2, \ldots , n\}$ such that $(x_k)(i)=j$, that is $x_kx_iH=x_jH$. If $\circ_T$ denotes the induced left quasigroup operation on $T$, then $x_k \circ_T x_i=x_j$. Suppose that $l=\sigma(x_k)(\alpha (i))$ where $l \in \{1,2, \ldots , n\}$. Then $x_lH=\sigma(x_k)x_{\alpha(i)}H=\sigma(x_k)\sigma(x_i)H$. Suppose that  $\circ_L$ denotes the operation in $L$. Then $x_lH=\sigma(x_k)  \circ_L \sigma(x_i)H=\sigma(x_k \circ_T x_i)H=\sigma(x_j)H$. This implies $l=\alpha(j)$. This proves the claim.

Now, since $\sigma(x_1)H=x_1H=y_1H$, so $\alpha(1)=1$. 
Thus $\alpha \in \text{Sym}(n-1)$. Converse is easy to show. 
\end{proof}

Now onwards we will use $K$ for the set $\{ \alpha \in \text{Sym}(n-1) \mid \text{there exist}~ T,L \in {\cal T}(G,H) ~\text{such that}~ \alpha T \alpha^{-1} =L \}$. Note that $K$ need not be a group. 

\begin{corollary}\label{2}
Let $T \in {\cal T}(G,H)$. Then $\text{Aut}(T) \subseteq K$.
\end{corollary}

Suppose that $A,B \leq \text{Sym}(n)$. Then by $N_B(A)$, we mean
$B \cap N_{\text{Sym}(n)}(A)$.

\begin{corollary} \label{3}
Suppose that $(G,H)$ satisfies one of the following conditions:

\n (i) $N_{\text{Sym}(n-1)}(G)=\text{Sym}(n-1)$.

\n (ii) Each left transversal of $H$ in $G$ generates $G$.

Then $K= N_{\text{Sym}(n-1)}(G) $.
\end{corollary}

\begin{proof}
We will first show that $N_{\text{Sym}(n-1)}(G) \subseteq N_{\text{Sym}(n-1)}(H)$. Suppose that $\alpha \in \text{Sym}(n-1)$ such that $\alpha G \alpha^{-1} =G$. Then $\alpha TH \alpha ^{-1} = G $ for any $T\in {\cal T}(G,H)$. That is $ \alpha T \alpha^{-1} \alpha H \alpha^{-1}=G$. Since $H \subseteq \text{Sym}(n-1)$,  $\alpha H \alpha^{-1} \subseteq G\cap \text{Sym}(n-1) =H$. Thus $ \alpha H \alpha^{-1} =H$. Thus $ \alpha \in N_{\text{Sym}(n-1)}(H)$. 

Clearly, if $\alpha \in N_{\text{Sym}(n-1)}(G)$, then $i_\alpha $ is an automorphism of $G$ which fixes $H$ (for $N_{\text{Sym}(n-1)}(G) \subseteq N_{\text{Sym}(n-1)}(H)$). Thus $i_\alpha$ maps left transversals of $H$ in $G$ to left transversals of $H$ in $G$. By Lemma \ref{1}, $\alpha \in K$. Hence $N_{\text{Sym}(n-1)}(G) \subseteq K$. Also $K \subseteq \text{Sym}(n-1)$. Since in case (i) $N_{\text{Sym}(n-1)}(G) = \text{Sym}(n-1)$, so for this case $K=N_{\text{Sym}(n-1)}(G)$.

Further, suppose that each left transversal of $H$ in $G$ generates $G$. Take $\alpha \in K$. So there exist $T, L \in {\cal T}(G,H)$ such that $\alpha T \alpha^{-1}=L$. Now $\alpha G \alpha^{-1}= \alpha \langle T \rangle \alpha^{-1}= \langle \alpha T \alpha^{-1} \rangle =\langle L \rangle =G$. Thus $\alpha \in N_{\text{Sym}(n-1)}(G)$. 


\end{proof}

\begin{example}
There are pairs $(G,H)$ for which all transversals generate the groups. For example, take $G$ to be a finite simple group and $H$  a subgroup of $G$ of order $2$. 
\end{example}

Following result is proved by Prof. P.J. Cameron (see  \cite{pjc}). 

\begin{lemma} \label{4a}
Let $G$ be a finite group and $H$ be its core-free subgroup. Then there exists at least one left transversal (with identity) of $H$ in $G$ which generates the whole group.
\end{lemma}

There are pairs $(G,H)$ such that none of the left transversals of $H$ in $G$ generate the group $G$. For example take $ G=\{x_1,x_2,y \mid x_1^3= x_2^3= y^2=1, yx_1y^{-1}=x_1^{2},  yx_2y^{-1}=x_2^{2}, x_1x_2=x_2x_1 \} $
and $H=\langle x_1, y \rangle $. Note that $|G|=18,~|H|=6$ and $Core_HG=\langle x_1 \rangle$. To generate $G$ we need at least three non-trivial elements. But each transversal of $H$ in $G$ contains only two non-trivial elements. Thus no transversal of $H$ in $G$ generates $G$.  

Recall that under our assumption $H$ is a core-free subgroup of $G\subseteq Sym(n)$, where $n$ is the index of $H$ in $G$. Then by Main Theorem of \cite{pss}, $\text{ict}(G,H) \neq 1$. The following lemma gives another proof of this fact.   

\begin{lemma}
$\text{ict}(G,H) \neq 1$.
\end{lemma}
\begin{proof}
Suppose that $\text{ict}(G,H) = 1$. By Lemma \ref{4a}, there exists $S \in {\mathcal T}(G,H)$ such that $\langle S \rangle =G$. Take $T \in  {\mathcal T}(G,H)$. Since $S$ and $T$ are isomorphic, by Lemma \ref{1}, there exists $\alpha \in N_{Sym(n-1)}(G)$ such that $\alpha S \alpha ^{-1} =T$. This implies $T$ generates $G$ for all $T \in {\mathcal T}(G,H)$. By Corollary \ref{3}, $K=N_{Sym(n-1)}(G)$. Also it is clear from the proof of Corollary \ref{3}, that $ N_{Sym(n-1)}(G) \subseteq N_{Sym(n-1)}(H)$. Thus $N_{Sym(n-1)}(G)$ acts transitively on ${\mathcal T}(G,H) $. Then $|H|^{n-1}=|{\mathcal T}(G,H)|$ divides $N_{Sym(n-1)}(G)$. That is $|H|^{n-1}$ divides $(n-1)!$. Take a prime $p$ dividing order of $H$. Then $p^{n-1}$ divides $(n-1)!$. It is not possible because the maximum exponent of a prime dividing $(n-1)!$ must be less than $\frac{n-1}{p-1}$ (see \cite[Problem 7, p. 122]{burton}).  This proves the lemma.
\end{proof}

Let us denote by $\text{Aut}_HG$, the set of automorphisms of $G$ which fixes $H$ and by $C_G(H)$, the centralizer of $H$ in $G$.

\begin{lemma}\label{4}
$\text{Aut}_{H}G \cong N_{\text{Sym}(n-1)}(G) /C_{\text{Sym}(n-1)}(G)$.
\end{lemma}

\begin{proof}
Consider a map $\phi :N_{\text{Sym}(n-1)}(G) \rightarrow \text{Aut}_{H}G $ as $\phi (\alpha)=i_{\alpha}$, where $\alpha \in N_{\text{Sym}(n-1)}(G)$.  This map is onto for take $f \in \text{Aut}_{H}G$. Then $f$ maps a left transversal to another left transversal of $H$ in $G$ isomorphically. By Lemma \ref{4a}, there exists $T \in {\cal T}(G,H)$ such that $\langle T \rangle =G$. Suppose that $f(T)=L \in {\cal T}(G,H)$. So by Lemma \ref{1}, there exists $\alpha \in \text{Sym}(n-1)$ such that $f=i_{\alpha}|_T$, inner conjugation determined by $\alpha$.

Clearly, $i_\alpha (T)=L$. That is, $i_{\alpha} (\langle T \rangle )=\langle L \rangle \subseteq G$. This implies $i_\alpha (G)=G$. Hence $\alpha \in N_{\text{Sym}(n-1)}(G)$.
It is obvious that kernel of $\phi$ is $C_{\text{Sym}(n-1)}(G)$.
\end{proof}

\begin{remark} \label{5}
For a group $A$ and its subgroup $B$, it is easy to observe that $\text{Aut}_{\chi (B)}\chi (A) \cong (\text{Aut}_{B}A)/L$, where $L=\{f \in \text{Aut}_{B}A \mid f(g)g^{-1} \in Core_{A}(B) ~\text {for all}~ g \in A\}$ and $\chi$ is the permutation representation of $A$ on $A/^lB$.
\end{remark}
Let us fix some notations for the rest of the paper. We denote $N_{\text{Sym}(n-1)}(G)$ by $ \g$ and identity element of Sym$(n)$ by $()$. 
Let $C_i, ~ 1 \leq i \leq r$ denote the conjugacy classes in $\g$ and  $x_i$ denotes the representatives from $C_i$. Each element of $\g$ acts naturally on $\{ 1,2, \ldots , n\}$. We call it \textit{first action}. We denote by  $t_i$ the number of orbits of $x_i$ 
(under first action) of length greater than $1$ and by $1= \delta _1, \delta _2, \ldots , \delta _{k_i} $ the distinct fixed points of $x_i$ (under first action). 
Also define ${\mathcal A}_{i1}:=1$ for all $i$ and ${\mathcal A}_{ij} := |\{ q \in G \mid q(1)= \delta _j ~\text{and}~ x_iqx_i^{-1}=q \}|$ where $1 \leq i \leq r$ and $1 < j \leq k_i$.

\begin{theorem}\label{6}
Let $(G,H) $ be a pair such that $K=\g$. Then 
\[ \text{ict}(G,H)=\frac{1}{|\g|}\sum_{i=1}^r \big( |C_i||H|^{t_{i}} \prod_{j=1}^{k_j}{\cal A}_{ij}\big),\]   
\end{theorem}

\begin{proof}
By Lemma \ref{1}, two elements of ${\mathcal T}(G,H)$ are isomorphic if and only if they are conjugate by an element of $K=\g$.
Also inner conjugation determined by elements of $\g$ maps $H$ to itself. This implies $\g$ acts on ${\mathcal T}(G,H)$ through conjugation (we call it \textit{second action}) and $\text{ict}(G,H)$ is equal to the number of orbits of action. Thus by Theorem 1.7 A, p. 24 of \cite{dix}, 
$$ \sum_{g \in \g} |\text{Fix} (g)| =\text{ict}(G,H)|\g|$$ where $\text{Fix} (g)=\{ T \in {\mathcal T}(G,H) \mid g T g^{-1} =T \}$. 
Since $|\text{Fix}(g)|=|\text{Fix}(hgh^{-1})|$ for all $h \in \g$, \text{ict}$(G,H)= \frac{1}{|\g|} \sum_{i=1}^r |C_i||\text{Fix}(x_i)|$, $x_i \in C_i$, $1 \leq i \leq r$.


Note that each non identity element of a transversal $T \in {\cal T}(G,H)$ moves $1$ to different symbols from the set $B=\{2,3, \ldots $ $, n \}$. 
Also, for each $i \in B$, there exists a unique $a \in T$ such that $a(1)=i$. 
Suppose that $T=\{()=a_1,a_2, \ldots , a_n \}$ such that for $j \in B$, $a_j(1)=j$. Take $x \in \g$ such that $xTx^{-1}=T$. Further, if $x(i)=j$, then $xa_ix^{-1}=a_j$. Conversely, for $x \in \g $, we want to form a $T \in {\cal T}(G,H)$ such that $xTx^{-1}=T$. If $x(i)=j$ for some $i,j \in B$, then choose $a_i \in G$ such that $a_i(1)=i$. There are $|a_iH|=|H|$ ways to choose such $a_i$. Take this $a_i \in T$. This implies $a^l=x^la_ix^{-l} \in T$ for all $l \in \N$. Suppose that $O_i(x)$ denotes the orbit of $i$ under action of $x$ on $B$. Then $a^{|O_i(x)|}=a_i$. 
  So choice of $a_i$ corresponding to one $i \in B$ will decide choice of $|O_i(x)|$ many elements of $T$. Further, ${\cal A}_{ij}$ gives the choice of remaining elements of $T$. Clearly, $|\text{Fix}(x_i)|=|H|^{t_i} \times \prod_{j=1}^{k_i}{\mathcal A}_{ij}$.
This proves the theorem.
\end{proof}

\begin{remark} \label{6a}
Suppose that $(G,H)$ is a pair such that each pair of left transversals which do not generate $G$ are in the distinct orbits of $\g$ under second action. Then even if $K \neq \g$, the above theorem can be applied to determine \text{ict}$(G,H)$.   
\end{remark}  

\begin{lemma} \label{6b}
Let $(G,H)$ be a pair and $T \in {\cal T}(G,H)$ such that $T$ is a cyclic characteristic subgroup of $G$. Then $\g =N_{\text{Sym}(n-1)}(T)$.
\end{lemma}

\begin{proof}
Since $T$ is characteristic subgroup of $G$ and $\g$ acts through inner conjugation, $\g \subseteq  N_{\text{Sym}(n-1)}(T)$.
Suppose that $|T|=[G : H]=n$ and $T=\langle a \rangle $.
Then $N_{\text{Sym}(n)}(T)=T \rtimes M$ (semidirect product), where $M=\{ u_k \in \text{Sym}(n) \mid 1 \leq k \leq n ~\text{and} ~ (k,n)=1 \}$ and $u_k$ is a permutation which sends $i$ to $ki \mod n$ for each $i \in \{ 1,\ldots , n \}$. 
Then clearly $N_{\text{Sym}(n-1)}(T)=\{ a^{n-l+1}u_l \mid (l,n)=1 \}$.   
Without loss of generality, take $a=(1,2, \ldots , n)$ and take a $b \in H$. 
Since $T$ is characteristic subgroup, so $bab^{-1}=a^i$ for some $1 \leq i < n$.  
Now $ba^{l-1}H=ba^{l-1}b^{-1}H=a^{(l-1)i}H$ where $1 \leq l \leq n$. 
Since each element of coset $a^{l-1}H$ maps $1$ to $l$ and after applying $b$ from left on this coset each element start mapping $1$ to $i(l-1)+1 \mod n$, so $b(l)=i(l-1)+1 \mod n , ~ 1 \leq l \leq n$. 
 With this definition of $b$, it is easy to verify that, $(a^{n-k+1}u_k) b (a^{n-k+1}u_k)^{-1}=b$.
 This proves that $N_{\text{Sym}(n-1)}(T)$ centralizes $H$ and, hence $N_{\text{Sym}(n-1)}(T) \subseteq \g$.
\end{proof}

The proof of following corollary follows from the proof of Lemma \ref{6b}. 
\begin{corollary}\label{6c}
Let $(G,H)$ be a pair and $T \in {\cal T}(G,H)$ such that $T$ is a cyclic normal subgroup of $G$ and any another left transversal is not isomorphic to $T$. Then $\g =N_{\text{Sym}(n-1)}(T)$. 
\end{corollary}

\begin{lemma}\label{6d}
Let $(G,H)$ be a pair and $T \in {\cal T}(G,H)$ such that $T=\langle a \rangle$ is a cyclic normal subgroup of $G$ and any another left transversal is not isomorphic to $T$. Then with the notations of Lemma \ref{6b}, we have

\n (i) $|\g| =\phi (n)$, where $\phi$ is Euler-Phi function and $n=|T|$,

\n (ii) $\g=\{ u_{j^{-1}}a^{j-1} \mid (j,n)=1  \}$ is abelian,

\n  (iii) $F(u_{j^{-1}}a^{j-1})=\{ i \in \{1,2,\ldots , n\}\mid (i-1)=\frac{l.n}{(n,j^{-1}-1)},  l \in \N\cup \{0 \} \}$ where $j \neq 1$ and $F(u_{j^{-1}}a^{j-1})$ denotes the fixed points of action of $u_{j^{-1}}a^{j-1}$ on $\{1, \ldots , n \}$,

\n (iv) Number of orbits of action of $\langle u_{j^{-1}}a^{j-1} \rangle $ is equal to 

\n$ \frac{\sum_{i=1}^{o(j)} |F((u_{j^{-1}}a^{j-1})^i)|}{o(j)} $.
\end{lemma}

\begin{proof} By the Corollary \ref{6c}, it follows that 
$$\g=N_{\text{Sym}(n-1)}(\langle a \rangle )=\{ a^{n-j+1}u_j \mid (j,n)=1 \}.$$

Further, since $u_j^{-1}=u_{j^{-1}}$ where $j^{-1} $ is an inverse of $j$ in $U_n$, the group of units of $\Z_n$, and $u_j^{-1}a^i u_j =a^{j^{-1}i}$, so we have
 \begin{eqnarray*} 
\g    &=& \{ u_ju_j^{-1}a^{n-j+1}u_j \mid (j,n)=1 \}\\ 
& =& \{ u_ja^{j^{-1}(n-j+1)} \mid (j,n)=1  \} \\ 
&=& \{ u_ja^{nj^{-1}-1+j^{-1}} \mid (j,n)=1  \}\\
&=& \{ u_{j}a^{j^{-1}-1} \mid (j,n)=1  \}\\
&=& \{ u_{j^{-1}}a^{j-1} \mid (j,n)=1  \}.
\end{eqnarray*}

 It is easy to observe that $\g$ is an abelian group. 
Now we will determine the orbit and fixed point of first action of each element of $\g$.
Let $i,j  \in \{1,2, \ldots , n \}$ and $(j,n)=1$. Then 
\begin{eqnarray}
u_{j^{-1}}a^{j-1}(i) &=& u_{j^{-1}}(i+j-1)\nonumber  \\
&=& j^{-1}(i+j-1) \mod{n}  \nonumber \\
&=& ij^{-1} +1-j^{-1} \mod{n} \nonumber \\
&=& (i-1)j^{-1}+1 \mod{n}. \nonumber 
\end{eqnarray}

Thus by induction, $ (u_{j^{-1}}a^{j-1})^n(i)= (i-1)j^{-n}+1$. 
Clearly ($u_{j^{-1}}a^{j-1})^n= u_{j^{-n}}a^{j^n-1}$. 
If $o(j) =l$, then $(u_{j^{-1}}a^{j-1})^l=()$. 
This proves that $o(u_{j^{-1}}a^{j-1})= s$ (say) divides $l$.
Further, $(u_{j^{-1}}a^{j-1})^s (i)$ $= (i-1)j^{-s}+1=i $ for all $i$. 
This implies that $ (i-1)(j^{-s}-1)=0 \mod n $ for all $i$.
Thus $j^{-s}=1 \mod n$. That is $l\mid s$. This implies $s=l$.

Now, we will determine the fixed points of $u_{j^{-1}}a^{j-1} $ under first action for $j \neq 1$. 
Suppose that $i \in \{1,2, \ldots , n \}$ is a fixed point. 
Then $u_{j^{-1}}a^{j-1} (i)=i$. 
That is, $(i-1)j^{-1}+1=i \mod n$. 
This implies $(i-1)(j^{-1}-1)=0 \mod n$. 
Then for some $m \in \N$, we can write $(i-1)(j^{-1}-1)=mn$. 
That is, $\frac{(i-1)(j^{-1}-1)}{(n, (j^{-1}-1))}=\frac{m.n}{(n, j^{-1}-1)}$ where  $j \neq 1$ and  $(n,j^{-1}-1)$ denotes the greatest common divisor of $n$ and $j^{-1}-1$.
If F$ (u_{j^{-1} }a^{j-1})$ denotes the fixed points of $ u_{j^{-1}}a^{j-1}$, 
then $F ({u_{j^{-1}}}a^{j-1})=\{ i \in \{1,2,\ldots , n\}\mid (i-1)=\frac{l.n}{(n,j^{-1}-1)},   l \in \N\cup \{0 \} \}$ where $j \neq 1$. For $j=1$, $u_{j^{-1}}a^{j-1}=()$ and so $F(u_{j^{-1}}a^{j-1})=n$.

Now we will determine the number of orbits of action of  $\langle u_{j^{-1}}a^{j-1} \rangle$ on $\{1,2, \ldots , n \}$. If $m_j$ denotes the number of orbits of action, then by \cite[Theorem 1.7 A,  p. 24]{dix}
$$m_j= \frac{\sum_{i=1}^{o(j)} |F((u_{j^{-1}}a^{j-1})^i)|}{o(j)}.$$

\end{proof}

\section{Application of Basic Idea : Number of left loops}

In \cite{rps}, it is shown that each right loop of order $n$ is isomorphic to a right transversal of $\text{Sym}(n-1)$ in $\text{Sym}(n)$. 
On the same way one can show that each left loop of order $n$ is isomorphic to a left transversal of $\text{Sym}(n-1)$ in $\text{Sym}(n)$. 
Thus $\text{ict}(\text{Sym}(n),\text{Sym}(n-1))$ is the number of non-isomorphic right/left loops of order $n$. 
Since the pair $(\text{Sym}(n),$ $\text{Sym}(n-1))$ satisfies Corollary \ref{3} (i), so $K=\g=\text{Sym}(n-1)$. 
Recall that $\g \subseteq \text{Sym}(n)$ has a natural action on $\{1, \ldots , n \}$. 
We call it the \textit{first action} and also $\g$ acts on ${\cal T}(G,H)$ through conjugation. We call it the \textit{second action}. By Theorem \ref{6}, we need to determine \text{Fix}$(x)$ where $ x \in \g$. Suppose again that $C_1=\{()\},C_2, \ldots , C_r$ denote the distinct conjugacy classes of $\text{Sym}(n-1)$ and $x_i \in C_i$, $1 \leq i \leq r$ is a representative from each conjugacy class. For $x_1 = ()$, \text{Fix}$(x_1)=((n-1)!)^{n-1}$. For $i>1$, suppose that cycle structure of $x_i$ is $((\mu_{i1},l_{i1}) \ldots (\mu_{is},l_{is}), \underbrace{1, \ldots , 1}_{k_i~times})$, that is $x_i$ has $\mu_{ij}$ cycles of length $l_{ij} >1$, $1 \leq j \leq s$ and $x_i$ has $k_i$ fixed points also $\sum_{j=1}^s \mu_{ij}l_{ij}=n-k_i$ and total number of orbits of lengths greater than one is equal to $ \sum_{j=1}^s \mu_{ij}=t_i$ (say). Under these notations, for $i >1$, we have following Lemma:


\begin{lemma} \label{7}
Suppose that $(G,H)=(Sym(n),Sym(n-1))$ and $1 < i \leq r$. Then

$$ |\text{Fix} (x_i)|= 
((n-1)!)^{t_i} ((k_i-1)! \prod_{j=1}^{s}\mu_{ij}! l_{ij}^{\mu_{ij}})^{k_i-1}. 
$$
\end{lemma}

\begin{proof}
Since $x_i \in \text{Sym}(n-1), ~x_i(1)=1$.
If $S \in {\cal T}(\text{Sym}(n),$ $\text{Sym}(n-1))$, then except one element (identity permutation) of $S$ other $n-1$ elements will move symbol $1$ to different symbol out of $n-1$ symbols. 
We are to determine the fixed points of $x_i$ under second action of $x_i$ on  
${\cal T}(\text{Sym}(n),\text{Sym}(n-1))$. 
By the proof of Theorem \ref{6}, $|\text{Fix}(x_i)|=|\text{Sym}(n-1)|^{t_i} \times \prod_{j=1}^{k_i}{\mathcal A}_{ij}$. 
Suppose that $1= \delta_1, \ldots , \delta_{k_i} $ are fixed points of $x_i$. If $k_i=1$, ${\cal A}_{ij}=1$. Suppose that $k_i >1$. For $k_i>1$, define  $B_{ij}=\{q \in \text{Sym}(n) \mid q(1)=\delta_j ~\text{and}~ x_iqx_i^{-1}=q \}$. Thus for $k_i>1$, ${\mathcal A}_{ij}=|B_{ij}|$. Following observations are sufficient to conclude the lemma:

\n (i) A cycle in the cycle decomposition of $q$ which consists of a fixed points of $x_i$ (under first action) can not contain non-fixed points of $x_i$.

\n (ii) If the set of elements of $\text{Sym}(n)$ which fix symbols outside the set $\{1=\delta_1, \ldots , \delta_{k_i} \}$ and moves $1$ to $\delta_j$ is denoted by $C$, then $|C|=(k_i-1)!$.

\n (iii) An element $q \in B_{ij}$ is product of an element from the set $C$ and an element from the set $D$, the normalizer of $x_i$ in $\text{Sym}(n-k_i)$ (symmetric group formed by non-fixed symbols of $x_i$).

\n (v) $|D|= \mu_{i1}! \ldots \mu_{is}! l_{i1}^{\mu_{i1}} \ldots l_{is}^{\mu_{is}}$.

\n (vi) Thus ${\mathcal A}_{ij}=|C||D|=(k_i-1)! |D|$.

\n (vii) $\prod_{j=1}^{k_i} {\mathcal A}_{ij}=(|C||D|)^{k_i-1}$ for $A_{i1}=1$.

\n This proves the lemma.
\end{proof}

\begin{example}
Now we will calculate \text{ict}$(\text{Sym}(4), \text{Sym}(3))$.

 $ \begin{array}{ |l | c | c | c | l | c |}
    \hline
    x_i & |C_i| & t_i & k_i & {\mathcal A}_{ij}  & | C_i| \times | \text{Fix}(x_i)| \\ \hline
    x_1=() & 1 & 0 & 4 &  & 6^3=216  \\ \hline
    x_2=(2,3) & 3 & 1 & 2 & 1!2^1 =2 & 3 (3!)^12^1=36 \\ \hline
    x_3=(2,3,4) & 2 & 1 & 1 & 1~(j=1) & 2(3!)^1 1=12 \\ \hline    
  \end{array} $
\smallskip

This implies \text{ict}$(\text{Sym}(4), \text{Sym}(3))= \frac{216+36+12}{3!}=44$.
\end{example}

\begin{example}
Now we will calculate \text{ict}$(\text{Sym}(5), \text{Sym}(4))$. 

 $ \begin{array}{ |l | c | c | c | l | c |}
    \hline
    x_i & |C_i| & t_i & k_i & {\mathcal A}_{ij}  & | C_i| \times | \text{Fix}(x_i)| \\ \hline
    x_1=() & 1 & 0 & 5 &  & (4!)^4=24^4  \\ \hline
    x_2=(2,3) & 6 & 1 & 3 & 2!2 =4 &  (4!)^1.6.4^2 \\ \hline
    x_3=(2,3,4) & 8 & 1 & 2 & 3 & (4!)^1.8.3^1 \\ \hline
    x_4=(2,3)(4,5) & 3 & 2 & 1 & 1~(j=1) & (4!)^2.3.1 \\ \hline
    x_5=(2,3,4,5) & 6& 1& 1&1~(j=1)& (4!)^1.6.1    \\ \hline 
  \end{array} $
\smallskip

\n This implies \text{ict}$(\text{Sym}(5), \text{Sym}(4))= $$\frac{24^2+24^2.4+24^2+24^2.3+24.6}{24}$$=$$14022$.
\end{example}

\section{Application of Basic Idea: $\text{ict}(\text{Alt}(n), \text{Alt}({n-1}))$ }

Consider that $(G,H)= (\text{Alt}(n), \text{Alt}({n-1}))$. Then $\g =\text{Sym}(n-1)$ and by Corollary \ref{3}, $\g=K$. Thus Theorem \ref{6} can be applied to determine $\text{ict}(\text{Alt}(n), \text{Alt}({n-1}))$. Under notations introduced before Lemma \ref{7}, we need to determine $\text{Fix}(x_i), ~ 1 < i \leq r$. For $x_1=(),~|\text{Fix}(x_1)|=|\text{Alt}(n-1)|^{n-1}$. For $x_i$ such that $k_i=1$, $|\text{Fix}(x_i)|=|\text{Alt}(n-1)|^{t_i}$. By $\text{Sym}(n-k_i)$, we mean the symmetric group formed by the set of non-fixed symbols of $x_i$.
\begin{lemma} \label{7a}
Suppose that $(G,H)= (\text{Alt}(n), \text{Alt}({n-1}))$ and $D$ denotes the normalizer of $x_i$ in $\text{Sym}(n-k_i)$. Then 

\n $ |\text{Fix} (x_i)|= 
\begin{cases}
(\frac{(n-1)!}{2})^{t_i} \frac{((k_i-1)! \prod_{j=1}^{s}\mu_{ij}! l_{ij}^{\mu_{ij}})^{k_i-1}}{2^{k_i-1}} & \text{if} ~k_i > 2     \\
 (\frac{(n-1)!}{2})^{t_i} \frac{( \prod_{j=1}^{s}\mu_{ij}! l_{ij}^{\mu_{ij}})^{k_i-1}}{2^{k_i-1}} & \text{if }~ k_i = 2~ \text{and} \\ &  D \not \subseteq \text{Alt}(n-k_i)
 \\ 0 & \text{if} ~ k_i = 2~ \text{and}\\ & D \subseteq \text{Alt}(n-k_i).   
 \end{cases} $ 
\end{lemma}

\begin{proof}
By Theorem \ref{6}, $|\text{Fix}(x_i)|= |\text{Alt}(n-1)|^{t_i}\prod_{j=1}^{k_i}{\cal A}_{ij}$.  
Suppose that $k_i=2$. Since in this case the elements of underlying set of ${\cal A}_{i2}$ will be of the form $(1,\delta_2)d$, where $d \in D$. But $(1,\delta_2)d \in \text{Alt}(n)$. This is possible if and only if $d$ is an odd permutation. Thus if $D \subseteq \text{Alt}(n-k_i)$, then ${\cal A}_{i2}=0$ and ${\cal A}_{ij}=|D|/2$ otherwise.
Suppose that $k_i >2$ and $1= \delta_1 , \ldots , \delta_{k_i} $ are fixed points of $x_i$. 
Then $q \in B_{ij}=\{q \in \text{Sym}(n) \mid q(1)=\delta_j ~\text{and}~ x_iqx_i^{-1}=q \} \subseteq \text{Alt}(n)$ is again product of cycles from the set $C$ and $D$ (where $C$ is the set of elements of $\text{Sym}(n)$ which fix symbols outside the set $\{1=\delta_1, \ldots , \delta_{k_i} \}$ and moves $1$ to $\delta_j$ and $D$ is the normalizer of $x_i$ in $\text{Sym}(n-k_i)$) in such a way that $q$ is even permutation. 
Clearly ${\cal A}_{ij}= \frac{|C|~|D|}{2}$. 
We know that $|C|= (k_i-1)!$ and $|D|=\mu_{i1}! \ldots \mu_{is}! l_{i1}^{\mu_{i1}} \ldots l_{is}^{\mu_{is}}$.
\end{proof}

\begin{example}
Consider $n=4$. We know that $\text{Sym}(3)$ has three conjugacy classes. The following table calculates the value of  
$\text{ict}(\text{Alt}(4), \text{Alt}(3))$.

 $ \begin{array}{ |l | c | c | c | l | c |}
    \hline
    x_i & |C_i| & t_i & k_i & {\cal A}_{ij}  & | C_i| \times | \text{Fix}(x_i)| \\ \hline
    x_1=() & 1 & 0 & 4 &  &3^3= 27  \\ \hline
    x_2=(2,3) & 3 & 1 & 2 & 1 & 9 \\ \hline
    x_3=(2,3,4) & 2 & 1 & 1 & 1~(j=1) & 6 \\ \hline    
  \end{array} $
\smallskip

This implies \text{ict}$(\text{Alt}(4), \text{Alt}(3))= \frac{27+9+6}{3!}=7$.
\end{example}

\begin{example}
Now we will calculate \text{ict}$(\text{Alt}(5), \text{Alt}(4))$. 

 $ \begin{array}{ |l | c | c | c | l | c |}
    \hline
    x_i & |C_i| & t_i & k_i & {\cal A}_{ij}  & | C_i| \times | \text{Fix}(x_i)| \\ \hline
    x_1=() & 1 & 0 & 5 &  & 12^4  \\ \hline
    x_2=(2,3) & 6 & 1 & 3 & 2  &  6.12.2^2=12^2.2 \\ \hline
    x_3=(2,3,4) & 8 & 1 & 2 & 0 & 0 \\ \hline
    x_4=(2,3)(4,5) & 3 & 2 & 1 & 1 & 3.12^2 \\ \hline
    x_5=(2,3,4,5) & 6& 1& 1 & 1 & 12.6  \\ \hline   
  \end{array} $
\smallskip

This implies \text{ict}$(\text{Alt}(5), \text{Alt}(4))= \frac{12^4+12^2.2+0+12^2.3+12.6}{4!}=897$.
\end{example}

\section{Application of Basic Idea: \text{ict}$(D_{n}, \langle b \rangle)$}

In this section, we will calculate a upper bound on the number of isomorphism class of transversals under some condition on the pair $(G,H)$ and as a particular case we determine it for a non-normal subgroup of order $p$ in a non-abelian group of order $pq$, where $q>p$ and $p$ and $q$ are prime numbers and  also for a non-normal subgroup of order two of a Dihedral group. 
\begin{lemma}
Let $(G,H)$ be a pair and $T\in {\cal T}(G,H)$ such that $T$ is a cyclic normal subgroup of $G$ and any another left transversal is not isomorphic to $T$. Then $$\text{ict}(G,H) \leq \sum_{(i,n)=1} |H|^{t_i + k_i-1}$$ where $t_i$ and $k_i$ are the number of orbits of length greater than one and number of fixed points of $u_{i^{-1}}a^{i-1} \in \g$ respectively as defined in Lemma \ref{6d}.
\end{lemma}

\begin{proof}
Let $T=\langle a \rangle $ be a left transversal of $H$ in $G$ which is a normal subgroup of $G$ of order $n$ and also no other left transversal is isomorphic to $T$. 
Then by Lemma \ref{6d}(ii), $\g=\{ u_{i^{-1}}a^{i-1} \mid (i,n)=1  \}$. 
Also $\g \subseteq K$, so the number of orbits under second action of $\g$ will only give an upper bound on $\text{ict}(G,H)$. To determine number of orbits of second action of $\g$, we use Theorem \ref{6} and with the notations of Theorem \ref{6}, we need to determine ${\cal A}_{ij}$, for each $i$ co-prime to $n$ and $2 \leq j \leq k_i$.

Let $1= \delta_1, \delta_2, \ldots , \delta_{k_i} $ be all the distinct fixed points of $x_i=u_{i^{-1}}a^{i-1}$. 
Then for $j \neq 1$, $${\cal A}_{ij} = |\{ q \in G \mid q(1)= \delta_j ~\text{and}~ u_{i^{-1}}a^{i-1}q (u_{i^{-1}}a^{i-1})^{-1}=q \}|.$$
Since above set is contained in some coset of $H $, so ${\cal A}_{ij} \leq |H|$. 
But it is easy to verify that $a^{\delta_j -1}h$, where $h \in H$ is an the element of underlying set of ${\cal A}_{ij}$.
Thus ${\cal A}_{ij}=|H|$ for all $i$ and $j>1$.
But \text{ict}$(G, H )\leq \frac {1}{\phi(n)} \sum_{(i,n)=1} |\text{Fix}(x_i)|$ and $x_i$'s are elements in $\g$. 
So, \text{ict}$(G, H)
\leq \frac {1}{\phi(n)} \sum_{(i,n)=1} |\text{Fix}(u_{i^{-1}}a^{i-1})|$ where 
$|\text{Fix}(u_{i^{-1}}a^{i-1})|$ $= |H |^{t_i} \times \prod_{j=1}^{k_i} {\cal A}_{ij}=|H|^{t_i+k_i-1}$, where  by Lemma \ref{6d}(iii), for $i>1$, $k_i=|F(x_i)|=|F( u_{i^{-1}}a^{i-1})|=|\{j \in \{1,2, \ldots ,n \} \mid j=1+ln /(n, i^{-1}-1) , l \in \N \cup \{ 0 \}  \}|$ and by Lemma \ref{6d}(iv), $m_i=k_i+t_i= \frac{\sum_{j=1}^{o(i)} |F(u_{i^{-1}}a^{i-1})^j|}{o(i)}$. 
 For $i=1$, $k_1=n$ and $t_1=0$.
Hence, $$\text{ict}(G, H ) \leq \frac{1}{\phi(n)} \sum_{(i,n)=1}|H|^{t_{i}} \times |H|^{k_i-1}=\frac{1}{\phi(n)} \sum_{(i,n)=1}(|H|)^{t_{i}+k_i-1}.$$
\end{proof}

The following corollary is an easy consequence of Remark \ref{6a} and above lemma.
\begin{corollary} \label{ict1}
Let $(G,H)$ be a pair and $T\in {\cal T}(G,H)$ such that $T$ is a cyclic normal subgroup of $G$ and any pair of left transversals which do not generate $G$ lie in the distinct orbits of $\g$. Then $$\text{ict}(G,H) = \sum_{(i,n)=1} |H|^{t_i + k_i-1}$$ where $k_1=n$ and $t_1=0$, for $i>1$,
$k_i=|\{j \in \{1,2, \ldots ,n \} \mid j=1+ln /(n, i^{-1}-1) , l \in \N \cup \{ 0 \}  \}|$ and $k_i+t_i=
\frac{\sum_{j=1}^{o(i)} k_{i^j}}{o(i)}$.
\end{corollary}

\begin{lemma}
Let $G$ be a non-abelian group of order $pq$ where $p$ and $q$ are prime numbers and $p<q$. Let $H$ be a subgroup of order $p$. Then 
$$ \text{ict} (G,H)= \frac{1}{q-1} \sum_{(j,q)=1} p^{t_j+k_j-1}$$
\end{lemma}
\begin{proof}
By Sylow theorem, $G$ has a unique subgroup $T$ of order $q$ and $T \in {\cal}(G,H)$. Thus except $T$ each  transversal of $H$ in $G$ generates $G$. Thus by Corollary \ref{ict1}, the lemma follows. 
\end{proof}

Let $D_n$ denotes the Dihedral group of order $2n$. Suppose that $H= \langle b  \rangle$ be a non-trivial normal subgroup of $D_n$ of order $2$. Then there exists $a \in D_n$ such that $$D_n= \langle a,b \mid a^n=1, b^2=1, ba=a^{n-1}b \rangle .$$ In $D_n$, a transversal of $\langle b \rangle $ either generates $D_n$ or is a subgroup. Under these notations following Lemma holds.

\begin{lemma} \label{8}
Let $T \in {\cal T} (D_n,\langle b \rangle)$ such that $T$ is a subgroup. Then

\n (i) $T= \langle a \rangle$, or 

\n (ii) $T= \langle a^2, ab \rangle \cong D_{n/2}$ (if $n$ is an even number).
\end{lemma}

\begin{proof}
It is easy to check that $\langle a \rangle$  is a transversal of $H$ in $G$. Each coset of $H$ in $G$ contains two elements $a^i$ and $a^ib$ for some $1 \leq i \leq n$. Suppose that $T$ is a subgroup and is different from $ \langle a \rangle$, that is $a \not \in T$. Then $ab \in T$. Suppose that $n$ is odd number. Then $(2,n)=1$, where $(2,n)$ denotes the greatest common divisor of $2$ and $n$.  Thus $a^2 \not \in T$. Therefore, $a^2b \in T$.  This gives $a=a^2b (ab)^{-1} \in T$, a contradiction. 

Further, suppose that $n$ is even. If $a^2b \in T$, then we come to a contradiction (as argued in the above paragraph). Thus $a^2 \in T$. Since $a^2 $ and  $ab \in T$ and $\langle a^2, ab \rangle$ is isomorphic to a Dihedral group of order $n~(=|T|)$, $T= \langle a^2, ab \rangle \cong D_{n/2}$. 
\end{proof}
The following lemma is an easy consequence of Lemma \ref{8} and Remark \ref{6a}.
\begin{lemma} \label{8a} For $(G,H)=(D_n, \langle b \rangle )$, Theorem \ref{6} is true even if $K \neq \g$.
\end{lemma}


The following lemma is a consequence of Corollary \ref{ict1} and the above lemma. 
\begin{lemma}
$$ \text{ict} (D_{n}, \langle b \rangle )= \frac{1}{\phi(n)} \sum_{(j,n)=1} 2^{t_j+k_j-1}$$
where $\langle b \rangle $ is a non-normal subgroup of $D_n$ of order $2$.
\end{lemma}

\begin{example}
Take $n=3$, $G=D_3$. Also $2^{-1}=2 \in U_3$. Now $i=1,2$. Thus $k_1=3,~ t_1=0$. For $i=2$. $k_2= |\{j \in \{1,2,3\}|j=1+3l,~ l \in \N \cap \{0\} \}=1$. Thus $k_2+t_2=\frac{k_2+k_1}{2}=2$. Hence, ict$(D_3, \langle b \rangle)=\frac{2^2+2^1}{2}=3$.

\end{example}

\begin{example}
Take $n=4$. Now $\phi(4)=2$, $G=D_4$. Here $i$ can attain two values $1$ and $3$. Further, $k_1=4,~t_1=0$, $k_3=|\{1, 3\}|=2$, $k_3+t_3=\frac{\sum_{i=1}^2(k_{3^i})}{2}=\frac{k_3+k_1}{2}=3$.
Thus ict$(D_4, \langle  b \rangle)=\frac{2^3+2^2}{2}=6$.

\end{example}

\textbf{ACKNOWLEDGMENTS}

\n The author wish to thank Dr. R.P. Shukla for giving useful suggestions for the improvement of the paper.

\end{document}